\documentclass[12pt]{amsart}

\usepackage[margin=1in]{geometry}
\usepackage{amssymb}
\usepackage{mathtools}
\usepackage{amsthm}
\usepackage{enumitem}
\usepackage{amsmath}
\usepackage[nobysame]{amsrefs}
\usepackage{varwidth}
\usepackage{relsize}
\usepackage{setspace}
\usepackage{adjustbox}
\usepackage{enumitem,linegoal}
\usepackage{calc}
\usepackage{mathrsfs}
\usepackage[mathscr]{eucal}
\usepackage[foot]{amsaddr}

\newtheorem{thm}{Theorem}

\newtheorem{corollary}{Corollary}
\newtheorem{proposition}{Proposition}
\newtheorem{thm*}{Theorem}
\newtheorem{proposition*}{Proposition}

\theoremstyle{definition}

\begin{document}

\title{Embeddings of Free Magmas with Applications to the Study of Free Non-associative Algebras}

\author{Cameron Ismail}

\address{Department of Mathematics, College of Charleston, 66 George Street, Charleston, South Carolina 29424, USA}



\begin{abstract}
We introduce an embedding of the free magma on a set $A$ into the direct product of the free magma on a singleton set and the free semigroup on $A$.  This embedding is then used to prove several theorems related to algebraic independence of subsets of the free non-associative algebra on $A$.  Among these theorems is a generalization of a result due to Kurosh, which states that every subalgebra of a free non-associative algebra is also free.
\end{abstract}

\maketitle

In universal algebra, a variety of algebras in which subalgebras of free algebras are free is called a \emph{Schreier variety}.  Examples of varieties which are Schreier varieties include the varieties of non-associative algebras, commutative algebras, anticommutative algebras, Lie algebras, and algebras with zero (trivial) multiplication \cite{umirbaev}.  The variety of all associative algebras is an example of a variety which is not Schreier \cite{cohn}.  The discovery of a complete description of all Schreier varieties remains an open problem, and it is this problem that primarily motivates the developments introduced here.  

Among the more recent developments in the study of Schreier varieties is the introduction of the notion of a PBW-pair of varieties by A.A. Mikhalev and I.P. Shestakov \cite{pbwpair}.  This concept generalizes the relationship between the variety of associative algebras and the variety of Lie algebras described by the Poincare-Birkhoff-Witt theorem \cite{shirshov,reutenauer}.  The study of primitive elements of a free algebra, which are elements contained in some free generating set for the free algebra, is also a subject of current research \cite{artamonov,primitive}.

In this paper, we describe an alternative approach to studying free non-associative algebras that is rooted in a particular embedding of a free magma.  After introducing this embedding in Section 2, we prove a few technical results and then, in Section 3, use our embedding to obtain direct sum decompositions of homogeneous subspaces of a free non-associative algebra.  These direct sum decompositions play a crucial role in Section 4, where we prove several results regarding algebraic independence.  Our key result is a generalization of Kurosh's Theorem \cite{kurosh}, which states that the variety of all non-associative algebras is Schreier.  

We point out that the ideas discussed in this paper are heavily-refined versions of those appearing in the author's master's thesis \cite{ismail}.

\section{Preliminaries}

Throughout our entire discussion, we will let $A$ denote a fixed, nonempty set and we will let $\mathscr{M}[A]$ denote the free magma on $A$.  Associated with each element $m$ of $\mathscr{M}[A]$ is a positive integer $\text{deg}(m)$ called the \emph{degree} of $m$.  If $\circ_M$ denotes the product defined on $\mathscr{M}[A]$, then the degree function on $\mathscr{M}[A]$ satisfies
\[ \text{deg}(m_1 \circ_M m_2) = \text{deg}(m_1) + \text{deg}(m_2) \]
for all $m_1, m_2 \in \mathscr{M}[A]$.  The elements of $A$ are the elements of $\mathscr{M}[A]$ with $\text{deg}(m) = 1$.  For $n > 1$, the elements $m$ of $\mathscr{M}[A]$ with $\text{deg}(m) = n$ are of the form $m_1 \circ_M m_2$ where $\text{deg}(m_1) + \text{deg}(m_2) = n$.  (See \cite{bourbaki_algebra}.)  The following proposition states one of the fundamental properties of the operation $\circ_M$.

\begin{proposition}\cite{bourbaki_algebra} \label{prop:magma1}
If $m_1, m'_1, m_2, m'_2 \in \mathscr{M}[A] $, then
\[ m_1 \circ_M m_2 = m'_1 \circ_M m'_2 \text{   if and only if   } m_1 = m'_1 \text{ and } m_2 = m'_2 \,. \]
\end{proposition}

\noindent We will refer to elements of $\mathscr{M}[A]$ as \emph{$A$-monomials}.  Rather than using the symbol $\circ_M$, we will from now on denote the product of two elements $m_1, m_2 \in \mathscr{M}[A]$ by $(m_1, m_2)$.  
\\

As a simple illustration, let $Z = \{z_1, z_2, z_3, z_4\}$ where $z_i \neq z_j$ if $i \neq j$.  The elements of $Z$ are elements of $\mathscr{M}[Z]$ of degree 1.  Products of the elements of $Z$,
\begin{equation} \label{deg2}
(z_1, z_1),\, (z_1, z_2),\, (z_1, z_3),\, (z_2, z_1),\, \ldots 
\end{equation}
are also in $\mathscr{M}[Z]$.  The elements in \eqref{deg2} have degree $2$.  Products of elements in \eqref{deg2} and elements of $Z$, 
\begin{equation} \label{deg3}
((z_1, z_2), z_3), \, (z_1, (z_2, z_3)), \, ((z_2, z_1), z_3), \, \ldots
\end{equation}
are in $\mathscr{M}[Z]$ and have degree $3$.  Products of elements in \eqref{deg3} and elements of $Z$ together with products of pairs of elements in \eqref{deg2} are the elements of $\mathscr{M}[Z]$ having degree $4$.  So, $\mathscr{M}[Z]$ contains the set $Z$, products of elements of $Z$, products of products of $Z$ and elements of $Z$, and so on.  
\\

The \emph{free semigroup} on the set $A$, denoted $\mathscr{S}[A]$, is the set of all finite sequences 
\[ f: \{1, 2, \ldots, n\} \to A \]
of elements of $A$ together with an operation $\circ_S$ defined as follows:
\\
\\
If $f_1 : \{1,2, \ldots, m\} \to A$ and $f_2 : \{1,2, \ldots, n\} \to A$ are elements of $\mathscr{S}[A]$, then 
\[f_1 \circ_S f_2 : \{1,2, \ldots, m + n\} \to A\]
is such that
\[ (f_1 \circ_S f_2)(k) = f_1(k) \]
if $1 \leq k \leq m$ and 
\[ (f_1 \circ_S f_2)(k) = f_2(k - m) \]
if $m+1 \leq k \leq m+n$ \cite{bourbaki_algebra}.
\\

The operation $\circ_S$, which is associative, is called \emph{concatenation}.  It is standard notation to write, for example, the element $f: \{1, 2, 3\} \to Z$ with $f(i) = z_i$ for all $i$ as
\[ z_1 z_2 z_3 \,. \]
If
\[ f: \{1, 2, \ldots, n\} \to A \]
is an element of $\mathscr{S}[A]$, then we say that the \emph{degree} of $f$ is $n$ and we write $\text{deg}(f) = n$.  

The following proposition, which is analogous to Proposition \ref{prop:magma1}, has a straightforward proof. 
\begin{proposition} \label{prop:semigroup1}
Let $f_1, f'_1, f_2, f'_2 \in \mathscr{S}[A] $.  If $\text{deg}(f_1) = \text{deg}(f'_1)$ or $\text{deg}(f_2) = \text{deg}(f'_2)$, then
\[ f_1 \circ_S f_2 = f'_1 \circ_S f'_2 \text{   if and only if   } f_1 = f'_1 \text{ and } f_2 = f'_2 \,. \]
\end{proposition}

For the full remainder of our discussion, we let $\mathbb{F}$ denote a fixed, arbitrarily chosen field.  If $\mathbb{F}[\mathscr{M}[A]]$ is the free $\mathbb{F}$-vector space over the set $\mathscr{M}[A]$, then let 
\[ \circ_N : \mathbb{F}[\mathscr{M}[A]] \times \mathbb{F}[\mathscr{M}[A]] \to \mathbb{F}[\mathscr{M}[A]] \]
be the unique bilinear product on $\mathbb{F}[\mathscr{M}[A]]$ which satisfies $m_1 \circ_N m_2 = m_1 \circ_M m_2$ for all $m_1, m_2 \in \mathscr{M}[A]$.  The vector space $\mathbb{F}[\mathscr{M}[A]]$ together with the operation $\circ_N$ is an $\mathbb{F}$-algebra, which we will denote by $\mathscr{N}[A]$.  $\mathscr{N}[A]$ is called the \emph{free (non-associative) $\mathbb{F}$-algebra on the set $A$} \cite{bourbaki_algebra}.  Rather than using the symbol $\circ_N$, we will from now on denote the product of two elements $p_1, p_2 \in \mathscr{N}[A]$ by $(p_1, p_2)$.  If $\{a_1, a_2, \ldots, a_n\}$ is a finite set of $n$ elements, then we will denote the free non-associative algebra $\mathscr{N}[\{a_1, a_2, \ldots, a_n\}]$ by $\mathscr{N}[a_1, a_2, \ldots, a_n]$.   

Under the obvious identification, $\mathscr{M}[A]$ is a subset of $\mathscr{N}[A]$ that is a basis for the vector space $\mathscr{N}[A]$.  We will call elements of $\mathscr{N}[A]$ $A$-polynomials.  Thus, every $A$-polynomial can be expressed uniquely as a linear combination of $A$-monomials.  For a nonzero $A$-polynomial $p$, if
\[ p = \sum_{i = 1}^k c_i n_i \]
for nonzero $c_i \in \mathbb{F}$ and $A$-monomials $n_i$, where $n_i \neq n_j$ if $i \neq j$, then the $A$-monomials $n_i$ are called the \emph{monomial terms} of $p$.  The \emph{degree} of an $A$-polynomial $p$, denoted $\text{deg}(p)$, is defined to be the degree of the highest-degree monomial term of $p$.  The degree of the zero $A$-polynomial is left undefined.  If $p_1$ and $p_2$ are any nonzero $A$ polynomials, then we have
\[ \text{deg}((p_1, p_2)) = \text{deg}(p_1) + \text{deg}(p_2) \,. \]

If, for example, $\mathbb{F} = \mathbb{Q}$ and $Z$ is as given previously, then
\[ p_1 = ((z_2, z_1), (z_4, z_4)) + 2(z_1, z_1) \,,\]
and
\[ p_2 = 4 (z_3, (z_1, z_1)) + z_2 + 3 z_3 \,,\]
are elements of the free $\mathbb{Q}$-algebra over the set $Z$.  We see that $\text{deg}(p_1) = 4$ and $\text{deg}(p_2) = 3$.
\\

For a positive integer $n$, let $\{X_1, \ldots, X_n\}$ be a finite set of $n$ indeterminates.  We will refer to elements of $\mathscr{M}[X_1, \ldots, X_n]$ and $\mathscr{N}[X_1, \ldots, X_n]$ as simply \emph{monomials} and \emph{polynomials}, respectively.

If $\mathscr{A}$ is an $\mathbb{F}$-algebra, then for any $n$-tuple $(y_1, \ldots, y_n)$ of elements $y_1, \ldots, y_n \in \mathscr{A}$, there is a unique algebra homomorphism 
\[ \phi_{(y_1, \ldots, y_n)} : \mathscr{N}[X_1, \ldots, X_n] \to \mathscr{A} \]
with 
\[ \phi_{(y_1, \ldots, y_n)} (X_i) = y_i \]
for all $1 \leq i \leq n$.  For a polynomial $P = P(X_1, \ldots, X_n)$, we will write
\[ \phi_{(y_1, \ldots, y_n)} (P) = P(y_1, \ldots, y_n) \,. \]
If $P$ is a polynomial, then we will write $P = P(X_1, \ldots, X_n)$ to express that $P$ is in $\mathscr{N}[X_1, \ldots, X_n]$.  

For a subset $G$ of $\mathscr{A}$, we will let $\langle G \rangle$ denote the subalgebra of $\mathscr{A}$ generated by $G$.  It is easy to verify that $\langle G \rangle$ can be described as the set of all elements of $\mathscr{A}$ of the form $P(g_1, \ldots, g_n)$,  where, for some positive integer $n$, $P \in \mathscr{N}[X_1, \ldots, X_n]$ and $g_1, \ldots , g_n \in G$.

\section{Sequence Type and Product Type}

The following theorem describes an embedding of the free magma $\mathscr{M}[A]$ into the direct product of two magmas which, unless $A$ is a singleton set, are both easier to work with than $\mathscr{M}[A]$.  An intuitive description of this embedding is given after the proof of the theorem.

\begin{thm} \label{thm:embedding}
Let $\mathscr{S}[A]$ be the free semigroup on $A$, let $\mathscr{M}[X]$ be the free magma on a singleton set $\{X\}$, and let $\pi_1$ and $\pi_2$ denote the canonical projections of the direct product of magmas $\mathscr{M}[X] \times \mathscr{S}[A]$ onto $\mathscr{M}[X]$ and $\mathscr{S}[A]$, respectively.  There is a unique injective magma morphism $\phi$ of the free magma $\mathscr{M}[A]$ into $\mathscr{M}[X] \times \mathscr{S}[A]$ such that 
\\

\begin{center}
\begin{varwidth}{\textwidth}
\begin{enumerate}[label=(\roman*)]
\item $\pi_1 \circ \phi$ is surjective and constant on the subset $A$ of $\mathscr{M}[A]$ and
\item $\pi_2 \circ \phi$ restricted to $A$ is the inclusion map $\iota_A: A \to \mathscr{S}[A]$. 
\\
\end{enumerate}
\end{varwidth}
\end{center}
Up to isomorphism, $\mathscr{M}[X]$ is the unique magma with this property.  
\end{thm}

\begin{proof}
Let $\phi : \mathscr{M}[A] \to \mathscr{M}[X] \times \mathscr{S}[A]$ be the unique magma morphism with $\phi(a) = (X, a)$ for all $a \in A$.  It is easy to see that $\phi$ satisfies the two given conditions.  A straightforward induction argument shows that
\begin{equation} \label{eq:embedding}
\text{deg}(m) = \text{deg}(\pi_1(\phi(m))) = \text{deg}(\pi_2(\phi(m))) 
\end{equation}
for all $m \in \mathscr{M}[A]$.

We now prove injectivity of $\phi$ by induction.  If $m, m' \in \mathscr{M}[A]$ both have degree 1 (so $m, m' \in A$) and $m \neq m'$, then clearly $\phi(m) \neq \phi(m')$.  Suppose that $n \geq 1$ is such that for any $m, m' \in \mathscr{M}[A]$ with $m \neq m'$ and $1 \leq \text{deg}(m), \text{deg}(m') \leq n$, we have $\phi(m) \neq \phi(m')$.  If $m, m' \in \mathscr{M}[A]$ are such that $m \neq m'$ and $1 \leq \text{deg}(m), \text{deg}(m') \leq n+1$, we have $m = (m_1, m_2)$ and $m' = (m'_1, m'_2)$ for some $m_i, m'_i \in \mathscr{M}[A]$ with $1 \leq \text{deg}(m_i), \text{deg}(m'_i) \leq n$.  Since $m \neq m'$, we have $m_1 \neq m'_1$ or $m_2 \neq m'_2$.  Without loss of generality, assume that $m_1 \neq m'_1$.  By the induction hypothesis, we have $\phi(m_1) \neq \phi(m'_1)$ and so $ \pi_1(\phi(m_1)) \neq \pi_1(\phi(m'_1)) $ or $\pi_2(\phi(m_1)) \neq \pi_2(\phi(m'_1))$.  If $ \pi_1(\phi(m_1)) \neq \pi_1(\phi(m'_1)) $, then 
\[ \pi_1(\phi(m)) = (\pi_1(\phi(m_1)), \pi_1(\phi(m_2))) \neq (\pi_1(\phi(m'_1)), \pi_1(\phi(m'_2))) = \pi_1(\phi(m')) \]
by Proposition \ref{prop:magma1}.  If $ \pi_1(\phi(m_1)) = \pi_1(\phi(m'_1)) $, then we must have 
\[ \text{deg}(\pi_2(\phi(m_1))) = \text{deg}(\pi_2(\phi(m'_1))) \]
by \eqref{eq:embedding}.  If, in addition, we have $ \pi_2(\phi(m_1)) \neq \pi_2(\phi(m'_1)) $, then 
\[ \pi_2(\phi(m)) = (\pi_2(\phi(m_1)), \pi_2(\phi(m_2))) \neq (\pi_2(\phi(m'_1)), \pi_2(\phi(m'_2))) = \pi_2(\phi(m')) \]   
by Proposition \ref{prop:semigroup1}.  In either case, we have $\phi(m) \neq \phi(m')$.  Injectivity of $\phi$ now follows easily.

We now show uniqueness of $\phi$.  If $\psi$ is any such function, then, because $\pi_1 \circ \psi$ is both surjective and constant on $A$,  $(\pi_1 \circ \psi)(A)$ is a one-element generating set for $\mathscr{M}[X]$.  This implies that $(\pi_1 \circ \psi)(A) = \{X\}$ or, equivalently, that $(\pi_1 \circ \psi)(a) = X$ for every $a$ in $A$.  Combining this together with the fact that $\pi_2 \circ \psi$ is the inclusion map from $A$ into $\mathscr{S}[A]$, we see that
\[ \psi(a) = ((\pi_1 \circ \psi)(a), (\pi_2 \circ \psi)(a)) = (X, a) \]
for each $a$ in $A$.  Therefore, we must have $\psi = \phi$.

Assume that there is a magma $M'$ that can be used in place of $\mathscr{M}[X]$ in Theorem \ref{thm:embedding}.  Let
\[ \phi' : \mathscr{M}[X] \to M' \times \mathscr{S}[X] \]
be the unique injective magma morphism described in the statement of the theorem.  It is easily verified that $\pi'_1 \circ \phi'$, where $\pi'_1$ is the canonical projection of $M' \times \mathscr{S}[X]$ onto $M'$, is a magma isomorphism of $\mathscr{M}[X]$ with $M'$.
\end{proof}

To illustrate, let $Z = \{z_1, z_2, z_3, z_4\}$, let $\pi_{Z,1}$ and $\pi_{Z,2}$ be the canonical projections of $\mathscr{M}[X] \times \mathscr{S}[Z]$ onto $\mathscr{M}[X]$ and $\mathscr{S}[Z]$, respectively, and let $\phi_Z$ be the embedding described in Theorem \ref{thm:embedding}.  If $m = (z_2, (z_3, (z_1, z_4)))$, then 
\[ (\pi_{Z,1} \circ \phi_Z)(m) = (X, (X, (X, X))) \, \in \mathscr{M}[X] \]
and
\[ (\pi_{Z,2} \circ \phi_Z)(m) =  z_2 z_3 z_1 z_4 \, \in \mathscr{S}[Z]. \]
This example illustrates a general principle regarding the functions $\pi_1 \circ \phi$ and $\pi_2 \circ \phi$ appearing in Theorem \ref{thm:embedding}.  If an element of $\mathscr{M}[A]$ is expressed symbolically as in this example, then the function $\pi_1 \circ \phi$ encodes the arrangement of parentheses and the function $\pi_2 \circ \phi$ encodes the sequence obtained if all parentheses are removed.  From now on, we will write $\pi_1 \circ \phi$ as $\Pi$ and $\pi_2 \circ \phi$ as $\Sigma$.  For a given $m$ in $\mathscr{M}[A]$, we will refer to $\Pi(m)$ and $\Sigma(m)$ as the \emph{product type} of $m$ and the \emph{sequence type} of $m$, respectively.

We point out that the core idea underlying the embedding of Theorem \ref{thm:embedding} is by no means a novel one.  If you regard, as a computer scientist might, the free magma on $A$ as the set of complete, planar, rooted binary trees with leaves labelled by elements of $A$, then for $m$ in $\mathscr{M}[A]$, $\Pi(m)$ gives the corresponding binary tree and $\Sigma(m)$ gives the labelling of leaves.  It is our hope that the framework provided by this embedding will lend itself well to effective computation in free non-associative algebras, and will ultimately enable us to discover a wealth of new Schreier varieties.  It is the author's view that the proofs of results presented in this paper serve as a testament to the promise of this approach.  
\\

The next proposition utilizes the notion of product type.  It will be used in Section 4.
\begin{proposition} \label{prop:differentpt_eval}
Let $m_1, \ldots, m_n \in \mathscr{M}[A]$ be of the same degree, and let $M = M(X_1, \ldots, M_n)$ and $M' = M'(X_1, \ldots, X_n)$ be monomials.  If $\Pi(M) \neq \Pi(M')$, then
\[ \Pi(M(m_1, \ldots, m_n)) \neq \Pi(M'(m_1, \ldots, m_n)) \,. \]
\end{proposition}

\begin{proof}
The proof is by induction on deg($M$).  The result is vacuously true if $\text{deg}(M) = 1$.  For some $k > 1$, assume that the result holds if $\text{deg}(M) = j$ for any $j$ with $1 \leq j < k$ and assume now that $\text{deg}(M) = k$.  If $M'$ is a monomial of degree 1, then the result is trivial.  If $M'$ is of degree greater than one and if $M = (M_1, M_2)$ and $M' = (M'_1, M'_2)$ for monomials $M_i, M'_i$, then (since $\Pi(M) \neq \Pi(M')$) we must have $\Pi(M_1) \neq \Pi(M'_1)$ or $\Pi(M_2) \neq \Pi(M'_2)$.  Without any loss of generality, assume that $\Pi(M_1) \neq \Pi(M'_1)$.  By the induction hypothesis, we have $\Pi(M_1(m_1, \ldots, m_n)) \neq \Pi(M'_1(m_1, \ldots, m_n))$.  Since
\[ \Pi(M(m_1, \ldots, m_n)) = (\Pi(M_1(m_1, \ldots, ,m_n)), \Pi(M_2(m_1, \ldots, m_n))) \]
and
\[ \Pi(M'(m_1, \ldots, m_n)) = (\Pi(M'_1(m_1, \ldots, m_n)), \Pi(M'_2(m_1, \ldots, m_n))) \,, \]
Proposition \ref{prop:magma1} implies that 
\[ \Pi(M(m_1, \ldots, m_n)) \neq \Pi(M'(m_1, \ldots, m_n)) \,. \]
\end{proof}

\section{Direct Sum Decompositions}

For each positive integer $n$, let $\mathscr{H}_n[A]$ be the subspace of $\mathscr{N}[A]$ generated by the $A$-monomials of degree $n$.  We have
\begin{equation*}
\mathscr{N}[A] = \bigoplus_{n = 1}^{\infty} \mathscr{H}_n[A] \,, 
\end{equation*}
and $\mathscr{N}[A]$ is a graded $\mathbb{F}$-algebra under this decomposition.  For each $n$,  $\mathscr{H}_n[A]$ is called the \emph{homogeneous component of $\mathscr{N}[A]$ of degree $n$}, and nonzero elements of $\mathscr{H}_n[A]$ are called \emph{homogeneous $A$-polynomials of degree $n$} \cite{bourbaki_algebra}.
\\

Thus, if $Z = \{z_1, z_2, z_3, z_4\}$ and $\mathbb{F} = \mathbb{Q}$, then the $Z$-polynomial
\[ p_1 = (z_2, z_1) + 4 (z_1, z_4) + 5 (z_2, z_2) + (z_1, z_3) \]
is homogeneous of degree 2, and the $Z$-polynomial
\[ p_2 = (((z_1, z_3), z_4), z_2) + 2 (((z_3, z_2), z_2), z_2) + (z_3, (z_1, (z_2, z_2))) \]
is homogeneous of degree 4.  
\\

For the remainder of our work, for any $n$, $\pi_n$ will denote the canonical projection map
\[ \pi_n : \mathscr{N}[A] \to \mathscr{H}_n[A] \,. \]
For any $A$-polynomial $p$, $\pi_n(p)$ will be called the \emph{homogeneous component of $p$ of degree $n$}.

The following proposition states that homogeneity is preserved when homogeneous $A$-polynomials are substituted into a monomial $M$. 
\begin{proposition} \label{prop:eval_degree}
If $M(X_1, \ldots, X_n)$ is a monomial and if $h_1, \ldots, h_n$ are nonzero homogeneous $A$-polynomials, then $M(h_1, \ldots, h_n)$ is a homogeneous $A$-polynomial with
\[ \text{deg}(M(h_1, \ldots, h_n)) = \sum\limits_{i=1}^n \text{deg}(M_i(h_1, \ldots, h_n)) \,,\]
where $M_i = (\Sigma(M))(i)$ for all $i$.  In particular, if $\text{deg}(h_i) = k$ for all $i$, then $M(h_1, \ldots, h_n)$ is a homogeneous $A$-polynomial of degree $k(\text{deg}(M))$.
\end{proposition}

\begin{proof}
If $\text{deg}(M) = 1$, then the result is trivial.  Assume that for some $n \geq 1$, the result holds for all monomials of degree $n$ and suppose now that $\text{deg}(M) = n + 1$.  If $M'$ and $M''$ are monomials such that $M = (M', M'')$, then we have
\begin{align*}
\text{deg}(M(h_1, \ldots, h_n)) &= \text{deg}(M'(h_1, \ldots, h_n), M''(h_1, \ldots, h_n)) \\
&= \text{deg}(M'(h_1, \ldots, h_n)) + \text{deg}(M''(h_1, \ldots, h_n)) \\
&= \sum\limits_{i=1}^{d'} \text{deg}(M'_i(h_1, \ldots, h_n)) + \sum\limits_{i=1}^{n + 1 - d'} \text{deg}(M''_i(h_1, \ldots, h_n)) 
\end{align*}
where $d' = \text{deg}(M')$.  Under the concatenation operation on $\mathscr{S}[A]$, we have 
\[ (\Sigma(M'))(j) = (\Sigma(M))(j) \]
for all $j$ with $1 \leq j \leq d'$ and
\[ (\Sigma(M''))(j) = (\Sigma(M))(d' + j) \]
for all $j$ with $1 \leq j \leq n + 1 - d'$.  So, 
\[ \sum\limits_{i=1}^{d'} \text{deg}(M'_i(h_1, \ldots, h_n)) =  \sum\limits_{i=1}^{d'} \text{deg}(M_i(h_1, \ldots, h_n)) \]
and
\[ \sum\limits_{i=1}^{n + 1 - d'} \text{deg}(M''_i(h_1, \ldots, h_n)) =  \sum\limits_{i = d' + 1}^{n + 1} \text{deg}(M_i(h_1, \ldots, h_n)) \,. \]
Therefore,
\begin{align*}
\text{deg}(M(h_1, \ldots, h_n)) &= \sum\limits_{i=1}^{d'} \text{deg}(M'_i(h_1, \ldots, h_n)) + \sum\limits_{i=1}^{n + 1 - d'} \text{deg}(M''_i(h_1, \ldots, h_n)) \\
&= \sum\limits_{i=1}^{d'} \text{deg}(M_i(h_1, \ldots, h_n)) + \sum\limits_{i=d' + 1}^{n+1} \text{deg}(M_i(h_1, \ldots, h_n)) \\
&= \sum\limits_{i=1}^{n+1} \text{deg}(M_i(h_1, \ldots, h_n))
\end{align*}
and so the proposition holds by induction.
\end{proof}

A subalgebra $S$ of $\mathscr{N}[A]$ is said to be \emph{homogeneous} if for every positive integer $n$ and $p \in S$ we have $\pi_n(p) \in S$ \cite{bourbaki_algebra}.  The following proposition states that a subalgebra of $\mathscr{N}[A]$ that is generated by a set of homogeneous $A$-polynomials is a homogeneous subalgebra.

\begin{proposition} \cite{lewin} \label{prop:homog_subalgebra}
If $H$ is a set of homogeneous $A$-polynomials, then $\langle H \rangle$ is a homogeneous subalgebra of $\mathscr{N}[A]$.
\end{proposition}

\begin{proof}
Let $P = P(X_1, \ldots, X_n)$ be a nonzero polynomial and let $h_1, \ldots, h_n$ be elements of $H$.  Assume $n$ is a positive integer such that 
\[ \pi_n(P(h_1, \ldots, h_n)) \neq 0 \,. \]  
We have $P = \sum\limits_{i=1}^l c_i M_i$ for some nonzero $c_i \in \mathbb{F}$ and distinct monomials $M_1, \ldots, M_l$.  By re-indexing if necessary, we may assume that $s$ with $1 \leq s \leq l$ is such that $ \text{deg}(M_i(h_1, \ldots, h_n)) = n $ if and only if $1 \leq i \leq s$.  If $P' = \sum\limits_{i=1}^s c_i M_i$, then we have
\[ \pi_n(P(h_1, \ldots, h_N)) = P'(h_1, \ldots, h_n) \in \langle H \rangle \,. \]
\end{proof}

A direct sum decomposition of each homogeneous component $\mathscr{H}_n[A]$ of $\mathscr{N}[A]$ is afforded to us by the product type function $\Pi$.  If $n$ is a postive integer, then for any $t$ in $\mathscr{M}[X]$ with $\text{deg}(t) = n$, we denote by $\mathscr{H}_{n,t}[A]$ the subspace of $\mathscr{H}_n[A]$ generated by the $A$-monomials $m \in \mathscr{M}[A]$ with $\Pi(m) = t$.  If, for each positive integer $n$, $\mathscr{M}_n$ is the set of degree-$n$ elements of $\mathscr{M}[X]$, then we have
\begin{equation*} 
\mathscr{H}_n[A] = \bigoplus_{t' \in \mathscr{M}_n} \mathscr{H}_{n,t'}[A] \,. 
\end{equation*}
This direct sum decomposition will be referred to as the \emph{product type direct sum decomposition} of $\mathscr{H}_n[A]$.  
\\

To illustrate, if $t_1 = ((X, X), X)$ and $t_2 = ((X, (X, X)), X)$, then
\[ p_3 = ((z_2, z_2), z_1) - 4 ((z_1, z_2), z_1) \]
is an element of $\mathscr{H}_{3, t_1}[Z]$ and
\[ p_4 = 3((z_4, (z_4, z_4)), z_4) + 2((z_2, (z_3, z_1)), z_4) + 3((z_1, (z_2, z_3)), z_4) \]
is an element of $\mathscr{H}_{4, t_2}[Z]$.
\\

We now present a proposition that demonstrates the utility of the embedding described in Theorem \ref{thm:embedding}.  This proposition will play a key role in the proof of Theorem \ref{thm:algebra5}. 

\begin{proposition} \label{prop:algebra1}
Let $M = M(X_1, \ldots, X_n)$ be a monomial and let $p_1, \ldots, p_n \in \mathscr{N}[A]$ be linearly independent.  If 
\[ \gamma: \Pi^{-1}(\{\Pi(M)\}) \subset \mathscr{N}[X_1, \ldots, X_n] \to \mathscr{N}[A] \]
is such that 
\[ \gamma(M') = M'(p_1, \ldots, p_n) \]
for all $M' \in \Pi^{-1}(\{\Pi(M)\})$, then $\gamma$ is injective and 
\[ \gamma(\Pi^{-1}(\{\Pi(M)\})) \]
is a linearly independent subset of $\mathscr{N}[A]$.
\end{proposition}

\begin{proof}
We will prove the result by induction on $\text{deg}(M)$.  If $\text{deg}(M) = 1$, then 
\[ \Pi^{-1}(\{\Pi(M)\}) = \{X_1, \ldots, X_n\} \,, \]
\[ \gamma(X_i) = p_i \text{ for each $i$} \,, \]
and
\[ \gamma(\Pi^{-1}(\{\Pi(M)\})) = \{p_1, \ldots, p_n\} \,. \]
By assumption, $\{p_1, \ldots, p_n\}$ is a linearly independent set.  It is clear that $\gamma$ is injective.

Assume that for some $n > 1$, the result holds for all $M' = M'(X_1, \ldots, X_n)$ with $1 \leq \text{deg}(M') < n$.  Suppose now that $\text{deg}(M) = n$, and let $M_1$ and $M_2$ be the unique monomials such that $M = (M_1, M_2)$.  For each $i$, let 
\[ \gamma_i : \Pi^{-1}(\{\Pi(M_i)\}) \subset \mathscr{N}[X_1, \ldots, X_n] \to \mathscr{N}[A] \]
be such that 
\[ \gamma_i(M') = M'(p_1, \ldots, p_n) \]
for each $M' \in \Pi^{-1}(\{\Pi(M_i)\})$.  Since 
\[ \Pi^{-1}(\{\Pi(M)\}) = \left\{ (M', M'') : M' \in \Pi^{-1}(\{\Pi(M_1)\}) \text{ and } M'' \in \Pi^{-1}(\{\Pi(M_2)\}) \right\} \,, \]
we may define a function
\[ \phi_M : \Pi^{-1}(\{\Pi(M)\}) \to \mathscr{N}[A] \otimes \mathscr{N}[A] \]
by letting
\[ \phi_M ((M',M'')) = M'(p_1, \ldots, p_n) \otimes M''(p_1, \ldots, p_n) \]
for each $M' \in \Pi^{-1}(\{\Pi(M_1)\})$ and $M'' \in \Pi^{-1}(\{\Pi(M_2)\})$.  

Because $\text{deg}(M_i) < n$ for each $i$, it follows from the induction hypothesis that $\phi_M$ is injective and 
\[ \phi_M(\Pi^{-1}(\{\Pi(M)\})) \]
is a linearly independent subset of $\mathscr{N}[A] \otimes \mathscr{N}[A]$.  If 
\[ \mathscr{N}[A] \mathscr{N}[A] = \text{span}\left(\{(y_1, y_2) : y_1, y_2 \in \mathscr{N}[A]\}\right) \,, \]
then let 
\[ \phi: \mathscr{N}[A] \mathscr{N}[A] \to \mathscr{N}[A] \otimes \mathscr{N}[A] \]
be the unique linear isomorphism such that 
\[ \phi((m_1, m_2)) = m_1 \otimes m_2 \]
for all $m_1, m_2 \in \mathscr{M}[A]$.  We have 
\[ \gamma = \phi^{-1} \circ \phi_M \]
and thus $\gamma$ is injective and 
\[ \gamma(\Pi^{-1}(\{\Pi(M)\})) \]
is a linearly independent subset of $\mathscr{N}[A]$.
\end{proof}

To illustrate Proposition \ref{prop:algebra1}, let $Z = \{z_1, z_2, z_3, z_4\}$, $\mathbb{F} = \mathbb{Q}$, $p_1 = (z_1, z_2)$, and $p_2 = ((z_3, z_3), z_2)$.  If $M = (X_1, X_1)$, then
\[ \Pi^{-1}(\Pi(M)) = \{ (X_1, X_1), (X_1, X_2), (X_2, X_1), (X_2, X_2) \} \]
and $\gamma(\Pi^{-1}(\Pi(M)))$ consists of the elements 
\[ ((z_1, z_2), (z_1, z_2)) \,, \] 
\[ ((z_1, z_2), ((z_3, z_3), z_2)) \,, \]
\[ (((z_3, z_3), z_2), (z_1, z_2)) \,, \] 
and
\[ (((z_3, z_3), z_2), ((z_3, z_3), z_2)) \]
which are all distinct and linearly independent by Proposition \ref{prop:algebra1}.
\\

\section{Algebraic Independence}

Let $\mathscr{A}$ be an $\mathbb{F}$-algebra.  A subset $\alpha$ of $\mathscr{A}$ will be called \emph{algebraically independent} if for any distinct $y_1, \ldots, y_n \in \alpha$ and any polynomial $P = P(X_1, \ldots, X_n)$, 
\[ P(y_1, \ldots, y_n) = 0 \text{ implies } P = 0. \]
By considering homogeneous degree-$1$ polynomials, we see that every algebraically independent set is also linearly independent. 

The significance of algebraic independence is the following:

\begin{thm} \label{thm:ai_meaning}
A nonempty subset $\alpha$ of an $\mathbb{F}$-algebra $\mathscr{A}$ is algebraically independent if and only if the algebra homomorphism 
\[ \psi : \mathscr{N}[\alpha] \to \langle \alpha \rangle \]
with $\psi(y) = y$ for all $y \in \alpha$ is an isomorphism. 
\end{thm}

\begin{proof}
It is clear that $\psi$ is surjective.  We will show that $\alpha$ is algebraically independent if and only if $\psi$ is injective.

For any $y' \in \text{ker} \, \psi$, we have
\[ y' = P(y'_1, \ldots, y'_n) \]
for some polynomial $P = P(X_1, \ldots, X_n)$ and some $y'_1, \ldots, y'_n \in \alpha$.  It follows from this that $\psi$ is injective if and only if for any choice of distinct $y_1, \ldots, y_n \in \alpha$, the restriction of $\psi$ to the subalgebra $\langle y_1, \ldots, y_n \rangle$ of $\mathscr{N}[\alpha]$ is injective.

For distinct $y_1, \ldots, y_n \in \alpha$, let 
\[ \gamma: \langle y_1, \ldots, y_n \rangle \subset \mathscr{N}[A] \to \mathscr{N}[y_1, \ldots, y_n] \]
be the isomorphism with $\gamma(y_i) = y_i$ for all $i$ and let 
\[ \gamma_g : \mathscr{N}[y_1, \ldots, y_n] \to \mathscr{N}[X_1, \ldots, X_n] \]
be the isomorphism induced by the function 
\[ g: \{y_1, \ldots, y_n\} \to \{X_1, \ldots, X_n\} \] 
with $g(y_i) = X_i$ for all $i$ \cite{ismail}.  If 
\[ \phi_{(y_1, \ldots, y_n)} : \mathscr{N}[X_1, \ldots, X_n] \to \langle \alpha \rangle \]
is the homomorphism with $\phi_{(y_1, \ldots, y_n)}(X_i) = y_i$ for all $i$, then we see that $\phi_{(y_1, \ldots, y_n)} \circ \gamma_g \circ \gamma$ is injective if and only if $\phi_{(y_1, \ldots, y_n)}$ is injective.  Since $\phi_{(y_1, \ldots, y_n)} \circ \gamma_g \circ \gamma$ is the restriction of $\psi$ to the subalgebra $\langle y_1, \ldots, y_n \rangle$ of $\mathscr{N}[\alpha]$, we see that $\psi$ is injective if and only if $\phi_{(y_1, \ldots, y_n)}$ is injective.

It follows easily from the definition of algebraic independence that $\phi_{(y_1, \ldots, y_n)}$ is injective for any choice of distinct $y_1, \ldots, y_n \in \alpha$ if and only if $\alpha$ is an algebraically independent set.  
\end{proof}

Theorem \ref{thm:ai_meaning} implies that a subset $\alpha$ of an $\mathbb{F}$-algebra $\mathscr{A}$ is algebraically independent if and only if the subalgebra $\langle \alpha \rangle$ of $\mathscr{A}$ generated by $\alpha$ is free over $\alpha$.  Thus, our definition of algebraic independence is equivalent to the one given in \cite{umirbaev}.  Our definition closely resembles the definition of algebraic independence that is used in ring theory \cite{lang}.
\\

It is now that we present several key results involving algebraic independence in free non-associative algebras.  The first two results are well-known, and are both consequences of the fact that the variety of non-associative algebras possesses what is known as the \emph{Nielsen property} \cite{mikhalev} (see the paragraph following the statement of Theorem \ref{thm:algebra6}).  We give proofs of these results in order to demonstrate our alternative framework for studying free non-associative algebras.  In particular, Proposition \ref{prop:algebra1}, which makes use of the embedding given in Theorem \ref{thm:embedding}, is used in the final step of the proof of Theorem \ref{thm:algebra5}.

\begin{thm} \label{thm:algebra5}
Let $\alpha \subset \mathscr{N}[A]$ be a set of nonzero, homogeneous $A$-polynomials of the same degree.  If $\alpha$ is a linearly independent set, then $\alpha$ is also an algebraically independent set.
\end{thm}

\begin{proof}
Let $k \geq 1$ be such that every $h \in \alpha$ is a homogeneous $A$-polynomial of degree $k$,  let $h_1, \ldots, h_n$ be distinct elements of $\alpha$, and let $P = P(X_1, \ldots, X_n)$ be a nonzero polynomial.  We will show that $P(h_1, \ldots, h_n)$ is nonzero. 

For some $t \geq 1$, we have
\[ P(X_1, \ldots, X_n) = \sum\limits_{i = 1}^{t} c_i M_i(X_1, \ldots, X_n) \]
where $c_i \in \mathbb{F}$ are nonzero for all $i$ and $M_1, \ldots, M_t$ are distinct monomials.  It follows from the bilinearity of the product operation on $\mathscr{N}[A]$ together with Proposition \ref{prop:eval_degree} that each $M_i(h_1, \ldots, h_n)$ is a homogeneous $A$-polynomial of degree $k(\text{deg}(M_i))$.  By reindexing if necessary, we can assume that some $l$ with $1 \leq l \leq t$ is such that 
\[ \text{deg}(M_1) = \text{deg}(M_2) = \cdots = \text{deg}(M_l) = \text{max}\left(\text{deg}(M_1), \text{deg}(M_2), \ldots, \text{deg}(M_t)\right) \]
and $\text{deg}(M_j) < \text{deg}(M_1)$ for all $l < j \leq t$.  If $P_1 = \sum\limits_{i=1}^l c_i M_i $ and $P_2 = \sum\limits_{i=l + 1}^t c_i M_i $, then we have
\[ \pi_{k(\text{deg}(M_1))}(P_1(h_1, \ldots, h_n)) = P_1(h_1, \ldots, h_n) \]
and 
\[ \pi_{k(\text{deg}(M_1))}(P_2(h_1, \ldots, h_n)) = 0 \,. \]
It suffices to show that $P_1(h_1, \ldots, h_n)$ is nonzero.
\\

Starting now with $P_1$, we can assume (by reindexing if necessary) that some $l'$ with $1 \leq l' \leq l$ is such that
\[ \Pi(M_1) = \Pi(M_2) = \cdots = \Pi(M_{l'}) \]
and $\Pi(M_{j'}) \neq \Pi(M_1)$ for all $j'$ with $l' < j' \leq l$.  It follows from Proposition \ref{prop:differentpt_eval} that no monomial term of $M_{j'}(h_1, \ldots, h_n)$ for any $j'$ with $l' < j' \leq l$ has the same product type as any monomial term of $M_i(h_1, \ldots, h_n)$ for any $i$ with $1 \leq i \leq l'$.  Thus, if $ P'_1 = \sum\limits_{i=1}^{l'} c_i M_i $, then it suffices to show that $P'_1(h_1, \ldots, h_n)$ is nonzero.  Because $M_1, M_2, \ldots, M_{l'}$ are distinct and have the same product type, and because $h_1, \ldots, h_n$ are linearly independent, the fact that $P'_1(h_1, \ldots, h_n)$ is nonzero follows from Proposition \ref{prop:algebra1}.
\end{proof}

\vspace{3mm}

In light of the following theorem, we emphasize that for a nonzero $A$-polynomial $p$, $\pi_{deg(p)}(p)$ is the highest-degree nonzero homogeneous component of $p$.

\begin{thm} \label{thm:algebra6}
Let $\alpha$ be a set of nonzero $A$-polynomials and let
\[ \alpha' = \{ \pi_{deg(p)} (p) : p \in \alpha \} \,. \]
If $\alpha'$ is algebraically independent and if 
\[ \pi_{deg(p_1)}(p_1) \neq \pi_{deg(p_2)}(p_2) \]
for all $p_1, p_2 \in \alpha$ with $p_1 \neq p_2$, then $\alpha$ is algebraically independent. 
\end{thm}

In other words, if no two elements of $\alpha$ have the same highest-degree homogeneous component and if the subset of $\mathscr{N}[A]$ consisting of the highest-degree homogeneous components of elements of $\alpha$ is algebraically independent, then $\alpha$ is algebraically independent.  Such a set $\alpha$ is said to be \emph{reduced} \cite{mikhalev}.  Theorem \ref{thm:algebra6} states that every reduced set is also an algebraically independent set.  In the language of universal algebra, this is what is known as the Nielsen property.

\begin{proof}
Let $P = P(X_1, \ldots, X_n)$ be a nonzero polynomial and let $p_1, \ldots, p_n$ be distinct elements of $\alpha$.  We will show that $P(p_1, \ldots, p_n)$ is nonzero.  

For some $t \geq 1$, $P = \sum\limits_{i = 1}^{t} c_i M_i $ where $M_1, \ldots, M_t$ are distinct monomials and each $c_i \in \mathbb{F}$ is nonzero.  For each $i$ with $1 \leq i \leq t$, 
\[ M_i(p_1, \ldots, p_n) = M_i(\pi_{deg(p_1)}(p_1), \ldots, \pi_{deg(p_n)}(p_n)) + r_i \]
where $r_i$ is a (possibly zero) $A$-polynomial with 
\[ \text{deg}(r_i) < \text{deg}(M_i \left(\pi_{\text{deg}(p_1)}(p_1), \ldots, \pi_{\text{deg}(p_n)}(p_n)\right) \]
if $r_i$ is nonzero.\footnote{This is easily shown using the bilinearity of the product operation on $\mathscr{N}[A]$.}  By re-indexing if necessary, we may assume that some $s \geq 1$ is such that 
\[ \text{deg}\left(M_i\left(\pi_{deg(p_1)}(p_1), \ldots, \pi_{deg(p_n)}(p_n)\right)\right) = \text{deg}(P(p_1, \ldots, p_n)) \]
if and only if $1 \leq i \leq s$.  We have
\[ \pi_{\text{deg}(P(p_1, \ldots, p_n))} (P(p_1, \ldots, p_n)) = \sum\limits_{i = 1}^{s} c_i M_i \left (\pi_{deg(p_1)} (p_1), \ldots, \pi_{deg(p_n)}(p_n) \right) \]
and by algebraic independence of $\alpha'$, 
\[ \pi_{deg(P(p_1, \ldots, p_n))} (P(p_1, \ldots, p_n)) \neq 0 \,. \]
Thus, the highest degree homogeneous component of $P(p_1, \ldots, p_n)$ is nonzero which implies that $P(p_1, \ldots, p_n)$ is nonzero.  
\end{proof}

\vspace{3mm}

The next three results can be viewed, collectively, as a generalization of Kurosh's Theorem \cite{kurosh}.  Indeed, Kurosh's Theorem is a direct corollary of Theorem \ref{thm:algebra8}.

We will say that a subset $S$ of $\mathscr{N}[A]$ is of \emph{bounded degree} if there is a positive integer $N$ such that $deg(p) \leq N$ for all nonzero $p$ in $S$.  If $S$ is nonempty and of bounded degree, then we will denote the smallest such $N$ by $\text{deg}(S)$.  If $S$ is empty, then we define $\text{deg}(S) = 0$.

\begin{thm} \label{thm:algebra7}
Let $\alpha_1 \subset \mathscr{N}[A]$ be an algebraically independent set of homogeneous $A$-polynomials which is of bounded degree.  If $\alpha_2 \subset \mathscr{N}[A]$ is a linearly independent set of homogeneous $A$-polynomials of the same degree $k \geq \text{deg}(S)$ and if
\[ \langle \alpha_1 \rangle \cap \text{span}(\alpha_2) = \{ 0 \} \,, \]
then $\alpha = \alpha_1 \cup \alpha_2$ is an algebraically independent set.
\end{thm}

\begin{proof}
Let $h_1, \ldots, h_n$ be distinct elements of $\alpha = \alpha_1 \cup \alpha_2$ and let $P = P(X_1, \ldots, X_n)$ be a nonzero polynomial with $P = \sum\limits_{i=1}^t c_i M_i$ for distinct monomials $M_1, \ldots, M_t$ and nonzero $c_1, \ldots, c_t$ in $\mathbb{F}$.  We will show that $P(h_1, \ldots, h_n)$ is nonzero.  

If $h_i \in \alpha_1$ for all $i$ or if $h_i \in \alpha_2$ for all $i$, then the result holds by algebraic independence of $\alpha_1$ and $\alpha_2$. (Algebraic independence of $\alpha_2$ follows from Theorem \ref{thm:algebra5}.)  So, by re-indexing if necessary, suppose that $j$ with $1 \leq j < n$ is such that $h_1, \ldots, h_j$ are elements of $\alpha_1$ and $h_{j+1}, \ldots, h_n$ are elements of $\alpha_2$.  By considering the direct sum decomposition of $\mathscr{N}[A]$ into homogeneous components we may assume that for every $i$, $M_i(h_1, \ldots, h_n)$ is a homogeneous $A$-polynomial of degree $d \geq 1$.  

We will prove the result by induction on $d$.  Because at least one of the $h_i$ is in $\alpha_2$, we can assume that $d \geq k$.  Otherwise, we would have $M_i = M_i(X_1, \ldots, X_j)$ for each $i$, and the fact that $P(h_1, \ldots, h_n)$ is nonzero would then follow from the algebraic independence of $\alpha_1$.  If $d = k$, then the fact that $P(h_1, \ldots, h_n)$ is nonzero follows easily from the algebraic independence of $\alpha_1$, the linear independence of $\alpha_2$, and the assumption that \[ \langle \alpha_1 \rangle \cap \text{span}(\alpha_2) = \{ 0 \}. \]  

Suppose $l > k$ is such that $P(h_1, \ldots, h_n)$ is nonzero if $k \leq d < l$, and assume now that $d = l$.  Since $l > k$, for each $i$ with $1 \leq i \leq t$ we have $M_i = (M'_i, M''_i)$ for some monomials $M'_i$ and $M''_i$.  By re-indexing if necessary, we may assume that $r \geq 1$ is such that 
\[ M'_1(h_1, \ldots, h_n), \ldots, M'_r(h_1, \ldots, h_n) \] 
(which are all homogeneous $A$-polynomials by Proposition \ref{prop:eval_degree}) are all of the same degree $d' \geq 1$ and that 
\[ \text{deg}(M'_j(h_1, \ldots, h_n)) \neq d' \] 
for all $j > r$.  Under this assumption, every $A$-monomial term of 
\[ M_i(h_1, \ldots, h_n) = (M'_i(h_1, \ldots, h_n), M''_i(h_1,\ldots, h_n)) \]
is of the form $(m'_i, m''_i)$ for $A$-monomials $m'_i$ and $m''_i$ such that $\text{deg}(m'_i) = d'$ if and only if $i \leq r$.  It follows that for any $j > r$ and $i \leq r$, no $A$-monomial term of $M_j(h_1, \ldots, h_n)$ is an $A$-monomial term of $M_i(h_1, \ldots, h_n)$.  Thus, it suffices to show that the polynomial $P' = \sum\limits_{i=1}^r c_i M_i$ is such that $P'(h_1, \ldots, h_n)$ is nonzero.

Using bilinearity and re-indexing if necessary, for some polynomials $Q_i$ we may write
\[ P' = \sum\limits_{i=1}^r c_i M_i = \sum\limits_{i=1}^s (M'_i, Q_i) \]
where $s \leq r$, $M'_i \neq M'_j$ for $i \neq j$, and each $Q_i(h_1, \ldots, h_n)$ is a homogeneous polynomial of degree $d - d'$.  We may also assume that for some $u$ with $1 \leq u \leq s$, there is an $A$-monomial $m$ of degree $d - d'$ which is an $A$-monomial term of $Q_i(h_1, \ldots, h_n)$ if and only if $1 \leq i \leq u$.  Using bilinearity, we can now write
\[ P'(h_1, \ldots, h_n) = \sum\limits_{i=1}^s (M'_i(h_1, \ldots, h_n), Q_i(h_1, \ldots, h_n)) =  \left( \left(\sum\limits_{i=1}^u c'_i M'_i(h_1, \ldots, h_n)\right), m \right) + R \,, \]
where each $c'_i \in \mathbb{F}$ is nonzero and $R$ is a homogeneous $A$-polynomial of degree $d$ which shares no common $A$-monomial terms with $\left( \left(\sum\limits_{i=1}^u c'_i M'_i(h_1, \ldots, h_n)\right), m \right)$.  We see that if $P'(h_1, \ldots, h_n) = 0$, then we must have 
\[ \left( \left(\sum\limits_{i=1}^u c'_i M'_i(h_1, \ldots, h_n)\right), m \right) = 0. \]  
But, by the induction hypothesis, 
\[ \sum\limits_{i=1}^u c'_i M'_i(h_1, \ldots, h_n) \neq 0 \] 
and so 
\[ \left( \left(\sum\limits_{i=1}^u c'_i M'_i(h_1, \ldots, h_n)\right), m \right) \neq 0 , \]  
which implies that $P'(h_1, \ldots, h_n)$ is nonzero.  This shows that the result holds if $d = l$, which completes the proof by induction.
\end{proof}

\vspace{3mm}

Theorem \ref{thm:algebra8} gives a situation in which an algebraically independent subset of a subalgebra  $S$ of $\mathscr{N}[A]$ can be extended to an algebraically independent set which generates the entire subalgebra $S$.  The proof of Theorem \ref{thm:algebra8} makes use of the following fact that is easily verified:

\begin{center}
If $\{\alpha_n\}_{n=1}^{\infty}$ is a collection of algebraically independent subsets of an $\mathbb{F}$-algebra $\mathscr{A}$ and if $i \leq j$ implies $\alpha_i \subset \alpha_j$ for all $i$ and $j$, then $\displaystyle{\bigcup\limits_{n=1}^{\infty} \alpha_n}$ is an algebraically independent subset of $\mathscr{A}$.
\end{center}

\vspace{2mm}

\begin{thm} \label{thm:algebra8}
Let $S$ be a homogeneous subalgebra of $\mathscr{N}[A]$ and let $\alpha_S \subset S$ be an algebraically independent set of homogeneous $A$-polynomials which is of bounded degree.  If
\[ \{p \in S : deg(p) \leq \text{deg}(\alpha_S) \} \subset \langle \alpha_S \rangle \,, \]
then there exists an algebraically independent set $\alpha \subset S$ of homogeneous $A$-polynomials such that
\[ \alpha_S \subset \alpha \text{\quad and \quad} \langle \alpha \rangle = S \,. \]
\end{thm}

We remark that the recursive construction utilized in the following proof first appeared in Kurosh's paper \cite{kurosh}.  A similar construction was later used by Shirshov \cite{shirshov} to show that every subalgebra of a free Lie algebra is free.  

\begin{proof}  If $\langle \alpha_S \rangle = S$, then we can take $\alpha = \alpha_S$.  Otherwise, assume that $\langle \alpha_S \rangle \neq S$.  We will recursively define a set $\alpha$ that satisfies the theorem.  

Let $k_1 \geq 1$ be smallest such that $deg(p_1) = k_1$ for some $p_1 \in S - \langle \alpha_S \rangle$.  By hypothesis, $k_1 > \text{deg}(\alpha_S)$.  Let
\[ W_{k_1} = \text{span}(\{p \in S : deg(p) = k_1\}) \]
and let $\gamma_{k_1}$ be a basis for $W_{k_1}$.  If 
\[ W_{k_1} \cap \langle \alpha_S \rangle \neq \{0\} \,, \] 
assume that $\gamma_{k_1}$ extends a basis $\beta_{k_1}$ for $W_{k_1} \cap \langle \alpha_S \rangle$ to a basis for $W_{k_1}$.  If 
\[W_{k_1} \cap \langle \alpha_S \rangle = \{0\} \,, \] 
let $\beta_{k_1} = \emptyset$.  Note that in either case, every element of $\gamma_{k_1} - \beta_{k_1}$ is of degree $k_1$.  By Theorem \ref{thm:algebra7}, 
\[ \alpha_{k_1} = (\gamma_{k_1} - \beta_{k_1}) \cup \alpha_S \]
is an algebraically independent set.  We see that for all $p \in S$ with $deg(p) \leq k_1$, $p \in \langle \alpha_{k_1} \rangle$.  
\\

For some $n > 1$ assume that $\alpha_{k_1}, \alpha_{k_2}, \ldots, \alpha_{k_{n-1}}$, where $ 1 \leq k_1 < k_2 < \ldots < k_{n-1}$, are algebraically independent subsets of $S$ consisting of homogeneous $A$-polynomials and are such that
\\
\[ (1) \hspace{.4cm} \alpha_{k_{i}} \subset \alpha_{k_{j}} \text{ if } i \leq j \, \]
\\
and
\\
\[ (2) \hspace{.4cm} \text{ for all } p \in S \text{ with } deg(p) \leq k_{n-1} \,, \,\, p \in \langle \alpha_{k_{n-1}} \rangle \,. \]
\\
If $\langle \alpha_{k_{n-1}} \rangle = S$, then let $k_n = k_{n-1} + 1$ and let $\alpha_{k_n} = \alpha_{k_{n-1}}$.  Otherwise, assume that $\langle \alpha_{k_{n-1}} \rangle \neq S$, and let $k_n > k_{n-1}$ be the smallest integer such that $deg(p) = k_n$ for some $\displaystyle{p \in S - \langle \alpha_{k_{n-1}} \rangle}$.  Let
\[ W_{k_n} = \text{span}(\{p \in S : deg(p) = k_n\}) \]
and let $\gamma_{k_n}$ be a basis for $W_{k_n}$.  If \[ W_{k_n} \cap \langle \alpha_{k_{n-1}} \rangle \neq \{0\} \, \] 
assume that $\gamma_{k_n}$ extends a basis $\beta_{k_n}$ for $W_{k_n} \cap \langle \alpha_{k_{n-1}} \rangle$ to a basis for $W_{k_n}$.  If 
\[W_{k_n} \cap \langle \alpha_{k_{n-1}} \rangle = \{0\} \,, \] 
let $\beta_{k_n} = \emptyset$.  By Theorem \ref{thm:algebra7},
\[ \alpha_{k_n} = (\gamma_{k_n} - \beta_{k_n}) \cup \alpha_{k_{n-1}} \]
is an algebraically independent set.  We also have
\[ \alpha_{k_{n-1}} \subset \alpha_{k_n} \]
and for all $p \in S$ with $deg(p) \leq k_n$, $p \in \langle \alpha_{k_n} \rangle$.
\\

We have a sequence of algebraically independent sets of homogeneous $A$-polynomials such that
\[ \alpha_{k_1} \subset \alpha_{k_2} \subset \ldots \subset \alpha_{k_n} \subset \ldots \,, \]
and so 
\[ \alpha = \bigcup\limits_{n = 1}^{\infty} \alpha_{k_n} \]
is algebraically independent.  Clearly, $\langle \alpha \rangle = S$.  We also have $\alpha_S \subset \alpha_{k_1} \subset \alpha$.

\end{proof}

Letting $\alpha_S = \emptyset$ in the statement of Theorem \ref{thm:algebra8}, we see that every homogeneous subalgebra of $\mathscr{N}[A]$ is free.

\vspace{3mm}

The following theorem extends Theorem \ref{thm:algebra8}.

\begin{thm} \label{thm:algebra8}
Let $S$ be a subalgebra of $\mathscr{N}[A]$, and let $\alpha_S \subset S$ be a set of $A$-polynomials such that
\[ \alpha_{S'} = \{ \pi_{deg(p)} (p) : p \in \alpha_S \} \]
is algebraically independent and
\[ \pi_{deg(p_1)}(p_1) \neq \pi_{deg(p_2)}(p_2) \]
for all $p_1, p_2 \in \alpha_S$ with $p_1 \neq p_2$.  

If $\alpha_S$ is of bounded degree and such that
\[ \{p \in S : deg(p) \leq \text{deg}(\alpha_S) \} \subset \langle \alpha_S \rangle \,, \]
then there exists an algebraically independent set $\alpha \subset S$ such that
\[ \alpha_S \subset \alpha \text{\quad and \quad} \langle \alpha \rangle = S \,. \]
\end{thm}

\begin{proof}
If $\langle \alpha_S \rangle = S$, then we can take $\alpha = \alpha_S$.  Assume that $\langle \alpha_S \rangle \neq S$ and let 
\[ H = \{ \pi_{\text{deg}(p)}(p) : p \in S \text{ and } p \neq 0 \} \]
and $S' = \langle H \rangle$.  Proposition \ref{prop:homog_subalgebra} implies that $S'$ is a homogeneous subalgebra of $\mathscr{N}[A]$.  By Theorem \ref{thm:algebra8}, $\alpha_{S'}$ can be extended to an algebraically independent set $\alpha'$ of homogeneous $A$-polynomials with $\langle \alpha' \rangle = S'$.  For each $h \in \alpha'$, let $f(h) \in S$ be such that
\[ \pi_{\text{deg}(f(h))}(f(h)) = h \]
and such that $f(h) \in \alpha_S$ if $h \in \alpha_{S'}$.  (It is easy to verify that such an $f$ exists.)  Now, let $\alpha = f(\alpha')$.  We have $\alpha_S \subset \alpha$, and Theorem \ref{thm:algebra6} implies that $\alpha$ is an algebraically independent set.  To show that $\langle \alpha \rangle = S$, suppose that $S - \langle \alpha \rangle$ is nonempty and assume that $N$ is the smallest positive integer such that $p \in S - \langle \alpha \rangle$ for some $p$ with $\text{deg}(p) = N$.  It follows easily from the fact that $\langle \alpha' \rangle = S'$ that we must have $N > 1$.  

Let $h_1, \ldots, h_n \in \alpha'$ and the polynomial $P$ be such that
\[ P(h_1, \ldots, h_n) = \pi_{\text{deg}(p)}(p) \,. \]
If $p_1, \ldots, p_n$ are the elements of $\alpha$ with $p_i = f(h_i)$ for all $i$, then we have
\[ p - P(p_1, \ldots, p_n) = r_p \]
where $r_p$ is a nonzero polynomial with $\text{deg}(r_p) < \text{deg}(p) = N$.  Since $p_i \in \alpha$ for all $i$, it follows that 
\[P(p_1, \ldots, p_n) \in \alpha \]
which implies that $r_p \not\in \langle \alpha \rangle$.  This contradicts the minimality of $N$, and so we must have $\langle \alpha \rangle = S$.
\end{proof}

\vspace{3mm}

Applying Theorem \ref{thm:algebra8} with $\alpha_S = \emptyset$ gives us Kurosh's Theorem.

\begin{corollary}[Kurosh's Theorem]\label{kurosh}
For any subalgebra $S$ of $\mathscr{N}[A]$, there exists an algebraically independent set $\alpha \subset \mathscr{N}[A]$ with $\langle \alpha \rangle = S$.
\end{corollary}

\section{Concluding Remarks}

As stated in the introduction, a complete description of all Schreier varieties is our main goal.  We are naturally led to the study of free non-associative algebras satisfying certain relations, which may be obtained by forming quotient algebras of the free non-associative algebra.  As a first step, we are interested in studying free non-associative algebras satisfying relations which take the form of a finite set of homogeneous polynomials.  It is hoped that the framework established in this paper will enable us to generate a wealth of new examples of Schreier varieties.  

\section{Acknowledgements}
The author would like to first thank Dr. Elizabeth Jurisich, who served as his thesis advisor.  The developments described in this paper would not have been possible without her patience and support.  The author would also like to thank  Dr. Ben Cox and Dr. Iana Anguelova, who served as thesis committee members.


\end{document}